\documentclass[11pt]{article}
\usepackage[T1]{fontenc}
\usepackage{lmodern}
\usepackage[tmargin=1in,bmargin=1in,lmargin=1in,rmargin=1in]{geometry}

\usepackage{amsthm}
\usepackage{amsmath}
\usepackage{amssymb}
\usepackage{bm}
\usepackage{mathrsfs}
\usepackage[dvips]{graphicx}

\usepackage{algorithmic}
\usepackage{algorithm}

\usepackage{color}

\usepackage{url}

\usepackage{supertabular}

\newtheorem{thm}{Theorem}
\newtheorem{cor}[thm]{Corollary}
\newtheorem{lemma}[thm]{Lemma}

\theoremstyle{definition}

\theoremstyle{remark}

\numberwithin{thm}{section}

\DeclareMathAlphabet{\mathsfsl}{OT1}{cmss}{m}{sl}

\renewcommand{\phi}{\varphi}

\newcommand{\mtx}[1]{\bm{#1}}

\newcommand{\rank}{\operatorname{rank}}

\newcommand{\trace}{\operatorname{tr}}

\newcommand{\Proj}{\ensuremath{\mtx{\Pi}}}

\newcommand{\psdle}{\preccurlyeq}
\newcommand{\psdge}{\succcurlyeq}

\newcommand{\bx}{\boldsymbol{x}}

\newcommand{\by}{\boldsymbol{y}}
\newcommand{\bZ}{\boldsymbol{Z}}

\newcommand{\bz}{\boldsymbol{z}}
\newcommand{\bX}{\boldsymbol{X}}
\newcommand{\bT}{\boldsymbol{T}}

\newcommand{\bW}{\boldsymbol{W}}

\newcommand{\bO}{\boldsymbol{O}}
\newcommand{\bP}{\boldsymbol{P}}

\newcommand{\bU}{\boldsymbol{U}}
\newcommand{\bV}{\boldsymbol{V}}

\def\reals{\mathbb{R}}
\def\bx{\boldsymbol{x}}

\def\bS{\boldsymbol{S}}
\def\b0{\mathbf{0}}
\def\bP{\boldsymbol{P}}

\def\bU{\boldsymbol{U}}

\def\bC{\boldsymbol{C}}

\def\bA{\boldsymbol{A}}
\def\bY{\boldsymbol{Y}}
\def\bB{\boldsymbol{B}}
\def\bG{\boldsymbol{G}}

\def\calT{\mathcal{T}}

\def\bI{\mathbf{I}}

\def\tr{\mathrm{tr}}

\def\Sp{\mathrm{Sp}}

\usepackage{algorithmic}
\usepackage{algorithm}

\usepackage[utf8]{inputenc} 
\usepackage[T1]{fontenc}    
\usepackage{hyperref}       
\usepackage{url}            
\usepackage{booktabs}       
\usepackage{amsfonts}       
\usepackage{nicefrac}       
\usepackage{microtype}      

\title{Note}

\author{
Teng Zhang  
}
\title{Tightness of the semidefinite relaxation for orthogonal trace-sum maximization}

\begin{document}

\maketitle
\begin{abstract}
This paper studies an optimization problem on the sum of traces of matrix quadratic forms on $m$ orthogonal matrices, which can be considered as a generalization of the synchronization of rotations. While the problem is nonconvex, the paper shows that its semidefinite programming relaxation can solve the original nonconvex problems exactly, under an additive noise model with small noise in the order of $O(-m^{1/4})$, where $m$ is the number of orthogonal matrices. This result can be considered as a generalization of existing results on phase synchronization.
\end{abstract}

\section{Problem Setup}
This paper considers the problem of estimating $m$ orthogonal matrices $\bO_1, \cdots, \bO_m$ with $\bO_i\in\reals^{d_i\times r}$ from the optimization problem:
\begin{equation}\label{eq:tracesum_problem}
    \{\hat{\bO}_i\}_{i=1}^m=\max_{\bO_i\in\reals^{d_i\times r}, i=1,\cdots,m} \sum_{i,j=1, i\neq j}^m\tr(\bO_i^T\bS_{ij}\bO_j),\,\,\,\text{subject to $\bO_i^T\bO_i=\bI$ for all $1\leq i\leq m$}.
\end{equation}
This problem is called orthogonal trace-sum maximization \cite{Won2018} and has application in generalized canonical correlation analysis. If $d_1=\cdots=d_m=r$, then \eqref{eq:tracesum_problem} is reduced to the the little Grothendieck problem over the orthogonal group \cite{Bandeira20164}, which have wide applications such as multi-reference alignment \cite{Bandeira2017}, cryo-EM \cite{doi:10.1137/090767777,ZHANG2017159}, 2D/3D point set registration~\cite{7430328,Cucuringu2012, doi:10.1137/130935458}, and multiview structure from motion \cite{6374980,7785130,7789623}.

While the optimization problem \eqref{eq:tracesum_problem} is nonconvex and difficult to solve, Won et. al \cite{Won2018} studies its convex relaxation as follow. Let $D=\sum_{i=1}^md_i$, \[
\bS=.\begin{bmatrix}
    \mathbf{0} & \bS_{12} & \dots  & \bS_{1m} \\
    \bS_{21} & \mathbf{0} &  & \bS_{2m} \\
    \vdots &  & \ddots& \vdots \\
   \bS_{m1} & \bS_{m2} &\cdots  &\mathbf{0}
\end{bmatrix}\in\reals^{D\times D},  \,\,\,\text{and} \,\,\,{\bO}=\left[ \begin{array}{c}
{\bO}_1\\
\vdots\\
{\bO}_m\\
 \end{array} \right]\in\reals^{D\times r},
\]
then using $\bU=\bO\bO^T$, \eqref{eq:tracesum_problem} can be relaxed to the convex problem
\begin{equation}\label{eq:tracesum_relaxed}
\hat{\bU}=\max_{\bU\in\reals^{D\times D}, \bU=\bU^T}\langle\bS, \bU\rangle, \text{subject to $\bU\psdge \mathbf{0}, \bU_{ii}\psdle\bI, \tr(\bU_{ii})=r$.}
\end{equation}

In this work, we assume the additive noise model as follows: there exists $\{\bV_i\}_{1\leq i\leq m}, \{\bW_{ij}\}_{1\leq i\neq j\leq m}$ such that $\bV_i\in\reals^{d_i\times r}$, $\bV_i^T\bV_i=\bI$ for all $1\leq i\leq m$,
\begin{equation}\text{$\bS_{ij}=\bV_i\bV_j^T+\bW_{ij}$, and $\bW_{ij}=\bW_{ji}$.}\label{eq:model}\end{equation} In this model, $\bV_i\bV_j^T$ is considered as the clean signal and $\bW_{ij}$ is considered as the additive noise. This is a natural model for the generalized canonical analysis problem in \cite{Won2018}, and when $d_1=\cdots=d_m=r$, this is used to model the synchronization of rotations problem \cite{10.1093/imaiai/iat005,10.1093/imaiai/iat006}.

The main contribution of this work shows that if the noises $\bW_{ij}$ are small, then the solutions of \eqref{eq:tracesum_problem} and \eqref{eq:tracesum_relaxed} are equivalent in the sense that $\hat{\bU}=\hat{\bO}\hat{\bO}^T$.

The main result, Theorem~\ref{thm:main}, shows that the convex relaxation in \eqref{eq:tracesum_relaxed} provides a tractable algorithm for solving the original problem \eqref{eq:tracesum_problem} exactly.  While there exists similar  results for the problem of phase synchronizationin \cite{Bandeira2017,doi:10.1137/17M1122025}, their method can not be extended to our setting directly and this work presents the first such result on the orthogonal trace-sum maximization problem and the synchronization of rotation problem. Compared with the works on phase synchronization, this proof depends on a different optimality certificate in Lemma~\ref{lemma:equavalency}.


\subsection{Related Works}
The problem of orthogonal trace-sum maximization problem or synchronization of rotations can be considered as a generalization of the angular or phase synchronization, which estimates angles $\theta_1, \cdots, \theta_m \in[0,2\pi)$ from the observation of relative offsets $(\theta_i-\theta_j)$ mod $2\pi$. The problem has applications in cryogenic electron microscopy \cite{Singer2009AngularSB}, comparative biology~\cite{Gao2019}, and many others. To address this problem, Singer~\cite{Singer2009AngularSB} formulate the problem as an optimization problem as follows: let $x_k=e^{i\theta_k}$ for all $1\leq k\leq m$, it attempts to solve the nonconvex problem \begin{equation}\label{eq:angular}
\max_{\bx\in\mathbb{C}^m}\bx^*\bC\bx,\,\,\text{s.t. $|x_1|=\cdots=|x_m|=1$}.
\end{equation}
To solve \eqref{eq:angular}, two methods are proposed in \cite{Singer2009AngularSB}, and one of the method solves its convex relaxation 
\begin{equation}\label{eq:angular_relax}
\max_{\bX\in\mathbb{C}^{m\times m}}\tr(\bC\bX),\,\,\text{s.t. $\bX_{11}=\cdots=\bX_{mm}=1$ and $\bX\psdge\mathbf{0}$}.
\end{equation}
In fact, \eqref{eq:angular} and \eqref{eq:angular_relax} can be considered as the special case of \eqref{eq:tracesum_problem} and \eqref{eq:tracesum_relaxed} when $d_1=\cdots=d_m=r=2$.

There has been many works that attempts to establish algorithms with theoretical guarantees for \eqref{eq:angular}. For example, Bandeira et al. \cite{Bandeira2017} assumes that $\bX=\bz\bz^*+\sigma\bW$, where $\bz\in\mathbb{C}^m$ satisfies $|z_1|=\cdots=|z_m|=1$ and $\bW\in\mathbb{C}^{m\times m}$ a Hermitian Gaussian Wigner matrix. It shows that if $\sigma\leq \frac{1}{18}m^{\frac{1}{4}}$, then the solution of \eqref{eq:angular_relax} is a matrix of rank one, which is also the solution to \eqref{eq:angular} in the sense that $\bX=\bx\bx^*$. Alternatively, Liu et al. \cite{doi:10.1137/16M105808X} investigated a modified power method for the original problem \eqref{eq:angular} and shows that the algorithm succeeds when $\sigma=O(m^{\frac{1}{6}})$. In addition, \cite{doi:10.1137/16M110109X} proves that a generalized power method converges to solution of \eqref{eq:angular}  when $\sigma=O(m^{\frac{1}{4}})$. Using a more involved argument and a modified power method, Zhong and Boumal improved the bound in \cite{Bandeira2017} to $\sigma=O(\sqrt{\frac{m}{\log m}})$. In fact, this paper follows this line of works and solve the problem of  \eqref{eq:tracesum_problem}, based on it convex relaxation \eqref{eq:tracesum_relaxed}. 

There are works that solve phase synchronization without using the optimization problem \eqref{eq:angular}. \cite{doi:10.1137/18M1217644} studies the problem from the landscape of a proposed objective function and shows that the global minimizer is unique even when the associated graph is incomplete and follows from the Erd\"{o}s-R\'{e}nyi random graphs.   \cite{doi:10.1002/cpa.21750} proposes an approximate message passing (AMP) algorithm, and analyzes its behavior by identifying phases where the problem is easy, computationally hard, and statistically impossible.

A even more special case of \eqref{eq:angular}  is the synchronization over $\mathbb{Z}_2=\{1,-1\}$~\cite{10.1093/comnet/cnu050}, which assume that $x_i$ in \eqref{eq:angular} are real-valued and $x_i=\pm 1$. For this problem, \cite{Fei2019AchievingTB} shows that the solution of \eqref{eq:angular_relax} matches the minimax lower bound on the optimal Bayes error rate for original problem \eqref{eq:angular}.

If $d_1=\cdots=d_m=r>2$, it is called the problem of synchronization of rotations in some literature. \cite{10.1093/imaiai/iat006} studies it from the perspective of estimation on Riemannian manifolds, and derive the Cram\'{e}r-Rao bounds of synchronization, that is, lower bounds on the variance of unbiased estimators, and \cite{DBLP:journals/corr/BoumalC13} shows that a lower bound concentrates on its expectation. Distributed algorithms with theoretical guarantees on convergence are proposed in \cite{THUNBERG2017243}. The formulation \eqref{eq:tracesum_problem} has applications in  graph realization and point cloud registration, multiview Structure from Motion (SfM)  \cite{6374980,7785130,7789623},  common lines in Cryo-EM \cite{doi:10.1137/090767777},  orthogonal least squares~\cite{ZHANG2017159}, and 2D/3D point set registration~\cite{7430328}. \cite{pmlr-v97-gao19f} generalized  \eqref{eq:tracesum_problem} by assuming multi-frequency information, and develop a two-stage algorithm that leverages the additional information.  \cite{10.1093/imaiai/iat005} discusses a method to make the estimator in \eqref{eq:tracesum_problem}  more robust to outlying observations. Another robust algorithm based on maximum likelihood estimator is proposed in \cite{6760038}. As for the theoretical properties, \cite{Bandeira20164} considers \eqref{eq:tracesum_relaxed} as an approximation algorithm to solve \eqref{eq:tracesum_problem}, and studies its approximation ratio. However, we are not aware of works in the spirit of  \cite{Bandeira2017,doi:10.1137/16M105808X,doi:10.1137/16M110109X,doi:10.1137/17M1122025} that studies the effectiveness of algorithm in the additive noise model \eqref{eq:model}.

The studied problem can be considered as a special case of the generic
 synchronization problems, which recovers a vector of
 elements given noisy pairwise measurements of the relative elements $g_ug_v^{-1}$, where here we assume that elements are in the group of orthogonal matrices.   \cite{abbe2017group} studies the properties of weak recovery when the elements are from a generic compact group and the underlying graph of pairwise observations is the $d$-dimensional grid.   \cite{doi:10.1002/cpa.21750} proposes an approximate message passing (AMP) algorithm for solving synchronization problems over a class of compact groups. \cite{doi:10.1137/16M1106055} generates the estimation from compact groups to the class of Cartan motion groups, which includes the important special case of rigid motions by applying the compactification process. \cite{doi:10.1002/cpa.21760} discusses the performance of a projected power method to the problem where the elements are scalars are the observations are the modulo differences $x_i - x_j$ mod $m$, and establishes its theoretical properties. \cite{6875186}  assumes that measurement graph is sparse and there are corrupted observations, and show that minimax recovery rate depends almost exclusively on the edge sparsity of the measurement graph irrespective of other graphical metrics.

\subsection{Notation}
This work sometime divides a matrix $\bX$ of size $D\times D$ into $m^2$ submatrices, such that the $ij-$th block is a $d_i\times d_j$ submatrix. We use $\bX_{ij}$ or $[\bX]_{ij}$ to denote this submatrix. Similarly, some times we divide a matrix of $\bY\in\reals^{D\times r}$ or a vector $\by\in\reals^D$ into $m$ submatrices or $m$ vector, where the $i$-th component, denoted by $\bY_i$, $[\bY]_i$ or $\by_i$, $[\by]_i$, is a matrix of size $d_i\times r$ or a vector of length $m_i$.

For any matrix $\bX$, we use $\|\bX\|$ to represent its operator norm and $\|\bX\|_F$ to represent its Frobenius norm. In addition, $\bP_{\bX}$ represents an orthonormal matrix whose column space is the same as $\bX$, $\bP_{\bX^\perp}$ is an orthonormal matrix whose column space is the orthogonal complement of the column space of $\bX$, $\Proj_{\bX}=\bP_{\bX}\bP_{\bX}^T$ is the projector to the column space of $\bX$, and $\Proj_{\bX^\perp}$ is the projection matrix to the orthogonal complement of the column space of $\bX$. If $\bY\in\reals^{n\times n}$ is symmetric, we use $\lambda_1(\bY)\geq \lambda_2(\bY)\geq \cdots\geq \lambda_n(\bY)$ to denote its eigenvalues in descending order.

\section{Main result}

The main result of the paper is as follows:
\begin{thm}\label{thm:main}
Assume that there exists $\bW\in\reals^{D\times D}$ such that $\bW_{ii}=\mathbf{0}$, $\bW_{ij}=\bW_{ij}^T$, and $\bS_{ij}=\bV_i\bV_j^T+\bW_{ij}$ for all $1\leq i\neq j\leq m$, where $\bV_i\in\reals^{d_i\times r}$ and $\bV_i^T\bV_i=\bI$ for all $1\leq i\leq m$, then when $\bW$ is small in the sense that
\begin{equation}\label{eq:condition_deterministic}
m-1>2m\frac{2\left(\max_{1\leq i\leq m}\|[\bW{\bV}]_i\|_F+4\|\bW\|^2\sqrt{\frac{r}{m}}\right)}{m-4\|\bW\|\sqrt{r}}+2\left(\max_{1\leq i\leq m}\|[\bW{\bV}]_i\|_F+4\|\bW\|^2\sqrt{\frac{r}{m}}\right)+4\|\bW\|\sqrt{\frac{r}{m}},\end{equation}
then the solutions of \eqref{eq:tracesum_problem} and \eqref{eq:tracesum_relaxed} are equivalent in the sense that $\hat{\bU}_{ij}=\hat{\bO}_i\hat{\bO}_j^T$ for all $1\leq i,j\leq m$.
\end{thm}
The proof of Theorem~\ref{thm:main} will be presented in Section~\ref{sec:proof}. While the condition \eqref{eq:condition_deterministic} is rather complicated, we can apply a probablistic model and prove that it holds with high probability under the regime that the size of noise grows with $m$. In particular, we follow \cite{doi:10.1137/16M105808X,Bandeira2017,doi:10.1137/17M1122025} and use additive Gaussian noise model that $\bW_{ij}$ are i.i.d. sampled from $N(0,\sigma^2)$.
\begin{cor}\label{cor:probablistic}
If for all $i> j$, the entries of $\bW_{ij}\in\reals^{m_i\times m_j}$ are i.i.d. sampled from $N(0,\sigma^2)$ and $\bW_{ji}=\bW_{ij}^T$, then the condition in \eqref{eq:condition_deterministic} holds almost surely for $\sigma<\frac{1}{16r^{\frac{3}{4}}d^{\frac{1}{2}}} m^{1/4}$ 
as $d_1=\cdots=d_m=d$, $d, r$ are fixed and $m$ goes to infinity.
\end{cor}
We remark that when $d=r=2$, Corollary~\ref{cor:probablistic} means that the algorithm suceeds when $\sigma<\frac{1}{40}m^{\frac{1}{4}}$, and it recovers the result in~\cite[Lemma 3.2]{Bandeira2017} ($\sigma<\frac{1}{14}m^{\frac{1}{4}}$) up to a constant factor. While it does not match the rate $\sigma=O(\sqrt{\frac{m}{\log m}})$ proved in \cite{doi:10.1137/17M1122025}, we conjecture that  a more involved argument can be used to improve the upper bound of $\sigma$ in Corollary \ref{cor:probablistic}.

\begin{proof}
WLOG we may assume that $\bV_i=[\bI_{r\times r};\mathbf{0}_{(m_i-r)\times r}]$, that is, $\bV_i$ consists of a $r\times r$ identity matrix and a $m_i-r\times r$ zero matrix. Then calculation shows that $[\bW\bV]_i$ is a matrix of size $m_i\times r$ that can be written as $[\bG_{r\times r}; \mathbf{0}_{(m_i-r)\times r}]$, and  $\bG$ is i.i.d. sampled from$N(0,\sigma^2(m-1))$. As a result, applying~\cite[Proposition 2.2]{measure2005} we have 
\[
\Pr\left(\frac{1}{\sigma^2(m-1)}\|[\bW\bV]_i\|_F^2>r^2+\delta\right)\leq \left(\frac{r^2}{r^2+\delta}\right)^{-r^2/2} e^{-\delta/2}
\]
and a union bound over $1\leq i\leq m$ implies\[
\Pr\left(\frac{1}{\sigma^2(m-1)}\max_{1\leq i\leq m}\|[\bW\bV]_i\|_F^2>r^2+\delta\right)\leq m\left(\frac{r^2}{r^2+\delta}\right)^{-r^2/2} e^{-\delta/2}.
\]

In addition, $\bW$ can be generated by $\bW=\bW^{(1)}+\bW^{(2)}$, where $\bW^{[1]}$ is i.i.d. generated from $N(0,\sigma^2/2)$, $[\bW^{(2)}]_{ij}=[\bW^{(1)}]_{ji}^T$ for $i\neq j$, and $[\bW^{(2)}]_{ii}=-[\bW^{(1)}]_{ii}$. Then for both $\bW^{(1)}$ and $\bW^{(2)}$, their entries are i.i.d. sampled from $N(0,\sigma^2/2)$, and~\cite[Theorem II.13]{Szarek:survey} implies that 
\[
\Pr\left(\frac{\sqrt{2}}{\sigma}\|\bW^{(1)}\|\geq 2\sqrt{D}+t\right)=\Pr\left(\frac{\sqrt{2}}{\sigma}\|\bW^{(2)}\|\geq 2\sqrt{D}+t\right)<e^{-t^2/2},
\] 
and as a result,
\[
\Pr\left(\frac{1}{\sigma}\|\bW\|\geq \sqrt{2}(2\sqrt{D}+t)\right)<2e^{-t^2/2},
\] 
\end{proof}
plugging $t=\frac{1}{2}\sqrt{D}$ and $\delta=2\log m$, \eqref{eq:condition_deterministic} holds with probability at least $1-\left(\frac{r^2}{r^2+2\log m}\right)^{-r^2/2}-2e^{-D/8}$, if
\begin{equation}\label{eq:condition_deterministic1}
m-1>2m\frac{2\left((\sigma\sqrt{m}(r^2+2\log m))+48D\sigma^2\sqrt{\frac{r}{m}}\right)}{m-16\sigma\sqrt{Dr}}+2\left((\sigma\sqrt{m}(r^2+2\log m))+48D\sigma^2\sqrt{\frac{r}{m}}\right)+16\sigma\sqrt{\frac{Dr}{m}},\end{equation}
which can be verified to hold when $d_1=\cdots=d_m=d$, $d, r$ are fixed and $m$ goes to infinity.

\section{Proof of Theorem~\ref{thm:main}}\label{sec:proof}
The proof of Theorem \ref{thm:main} can be divided into several components as follows. First, Lemma~\ref{lemma:equivalent} establishes an equivalent formulation of \eqref{eq:tracesum_problem}, given in  \eqref{eq:tracesum_problem1}. Based on this equivalent formulation, Lemma~\ref{lemma:stationary} shows the property that $\bS$ can be decomposed into two matrices, each with a certain property. Then, based on this decomposition, Lemma~\ref{lemma:equavalency} gives a condition such that the minimizer of \eqref{eq:tracesum_problem1} is also the unique solution to \eqref{eq:tracesum_relaxed}. It remains to verify this condition. Lemma~\ref{lemma:clean} analyzes the decomposition for the clean case that $\bW=\mathbf{0}$. Then using a perturbation argument, Lemmas~\ref{lemma:noisy}, ~\ref{lemma:noisy2}, and~\ref{lemma:pertubation} shows that when $\|\bW\|$ is small, the decomposition does not change much and the condition would still be satisfied.

We first present our lemmas and a short proof of Theorem~\ref{thm:main}, and leave the proofs of lemmas to Section~\ref{sec:proof_lemma}. 
\begin{lemma}\label{lemma:equivalent}
The optimization problem \eqref{eq:tracesum_problem} is equivalent to the problem 
\begin{equation}\label{eq:tracesum_problem1}
\tilde{\bU}=\max_{\bU\in\reals^{D\times D}, \bU=\bU^T}\tr(\bS\bU), \text{subject to $\bU\psdge \mathbf{0}, \bU_{ii}\psdle\bI, \tr(\bU_{ii})=r, \rank(\bU)=r$.}
\end{equation}
\end{lemma}

\begin{lemma}\label{lemma:stationary}
The solution to \eqref{eq:tracesum_problem1} satisfied the following property: let $\tilde{\bU}=\tilde{\bV}\tilde{\bV}^T$ with $\tilde{\bV}\in\reals^{D\times r}$, 
then $\bS$ can be written as $\bS=\bT^{(1)}+\bT^{(2)}$, where $\bT^{(1)}, \bT^{(2)}\in\reals^{D\times D}$ satisfy that  $\Proj_{\tilde{\bV}^\perp}\bT^{(1)} \Proj_{\tilde{\bV}^\perp}=\bT^{(1)}$, $\bT^{(2)}_{ij}=\mathbf{0}$ for $1\leq i\neq j\leq m$, and $\Proj_{\tilde{\bV}_i}\bT^{(2)}_{ii} \Proj_{\tilde{\bV}_i}=\bT^{(2)}_{ii}$ for $1\leq i\leq m$.

In addition, $\bT^{(1)}$ and $\bT^{(2)}$ are defined as follows: let $\tilde{\bV}\in\reals^{D\times r}$ such that $\tilde{\bU}=\tilde{\bV}\tilde{\bV}^T$, then
\begin{equation}\label{eq:tildeT1}
[\bT^{(1)}]_{ij}=\bS_{ij},\,\, [\bT^{(2)}]_{ij}=\mathbf{0}\,\,\text{for $1\leq i\neq j\leq m$}, \,\,[\bT^{(1)}]_{ii}=-[\bT^{(2)}]_{ii}\,\,\text{for $1\leq i\leq m$},
\end{equation}
and $[\bT^{(1)}]_{ii}$ is chosen such that
\begin{equation}\label{eq:tildeT2}
P_{\tilde{\bV}_i^\perp}^T[\bT^{(1)}]_{ii}=\mathbf{0},\,\,[\bT^{(1)}]_{ii}P_{\tilde{\bV}_i^\perp}=\mathbf{0},\,\,\tilde{\bV}_i^T[\bT^{(1)}]_{ii}\tilde{\bV}_i=-\tilde{\bV}_i^T\left(\sum_{j=1,j\neq i}^m\bS_{ij}\tilde{\bV}_j\right).
\end{equation}

\end{lemma}

\begin{lemma}\label{lemma:equavalency}
If there exists a decomposition $\bS=\hat{\bT}^{(1)}+\hat{\bT}^{(2)}+c\bI$ such that $\Proj_{\tilde{\bV}^\perp}\hat{\bT}^{(1)}\Proj_{\tilde{\bV}^\perp}=\hat{\bT}^{(1)}$, $\hat{\bT}^{(2)}_{ii}=\Proj_{\tilde{\bV}_i}\hat{\bT}^{(2)}_{ii} \Proj_{\tilde{\bV}_i}$ for all $1\leq i\leq m$,  and $\{P_{\tilde{\bV}_i}^T\hat{\bT}^{(2)}_{ii} P_{\tilde{\bV}_i}\}_{i=1}^m$ and $-P_{\tilde{\bV}^\perp}^T\hat{\bT}^{(1)}P_{\tilde{\bV}^\perp}$ are  positive definite matrices. The $\tilde{\bU}$, the solution to \eqref{eq:tracesum_relaxed}, is also the unique solution to \eqref{eq:tracesum_problem1}.
\end{lemma}

\begin{lemma}\label{lemma:clean}
Denote $\bT^{(1)}$ generated by \eqref{eq:tildeT1} and \eqref{eq:tildeT2} under the case $\bW=\mathbf{0}$ as $\bT^{(1)*}$. Let $L_1=\Sp(\bV)$ and $L_2=\{\bx\in\reals^D: \bx_i\in\Sp(\bV_i)\}$, then  $\bT^{(1)*}=-m\Proj_{L_2\cap L_1^\perp}$. 
\end{lemma}

\begin{lemma}\label{lemma:noisy}
\begin{align}
&\|\bT^{(1)*}-\bT^{(1)}\|\leq m\max_{1\leq i\leq m}\|\tilde{\bV}_i-\bV_i\|+\max_{1\leq i\leq m}\|\sum_{j=1}^m\bW_{ij}\tilde{\bV}_j\|+\|\bV^T\tilde{\bV}-m\bI\|,\label{eq:noisy1}\\
&\text{$\|\tilde{\bV}_i^T[\bT^{(2)}]_{ii}\tilde{\bV}_i-(m-1)\bI\|<\Big\|\sum_{j=1}^m\bW_{ij}\tilde{\bV}_i\Big\|+m\|\tilde{\bV}_i^T\bV_i-\bI\|+\|{\bV}^T\tilde{\bV}-m\bI\|$ for all $1\leq i\leq m$\label{eq:noisy2}.}
\end{align}
\end{lemma}

\begin{lemma}\label{lemma:noisy2}
If \begin{equation}\label{eq:noisy2_0}m-1>2m\max_{1\leq i\leq m}\|\tilde{\bV}_i-\bV_i\|+2\max_{1\leq i\leq m}\|\sum_{j=1}^m\bW_{ij}\tilde{\bV}_j\|+2\|\bV^T\tilde{\bV}-m\bI\|,\end{equation} then there exists  $c$, $\hat{\bT}^{(1)}$ and $\hat{\bT}^{(2)}$ such that the conditions in Lemma~\ref{lemma:equavalency} are satisfied, that is, $\bS=\hat{\bT}^{(1)}+\hat{\bT}^{(2)}+c\bI$, $\Proj_{\tilde{\bV}^\perp}\hat{\bT}^{(1)}\Proj_{\tilde{\bV}^\perp}=\hat{\bT}^{(1)}$, $\hat{\bT}^{(2)}_{ii}=\Proj_{\tilde{\bV}_i}\hat{\bT}^{(2)}_{ii} \Proj_{\tilde{\bV}_i}$ for all $1\leq i\leq m$, and $\{P_{\tilde{\bV}_i}^T\hat{\bT}^{(2)}_{ii} P_{\tilde{\bV}_i}\}_{i=1}^m$ and $-P_{\tilde{\bV}^\perp}^T\hat{\bT}^{(1)}P_{\tilde{\bV}^\perp}$ are  positive definite matrices.
\end{lemma}

\begin{lemma}\label{lemma:pertubation}
The solution of \eqref{eq:tracesum_problem1}, $\tilde{\bU}$, admits a decomposition $\tilde{\bU}=\tilde{\bV}\tilde{\bV}^T$ with $\tilde{\bV}\in\reals^{D\times r}$, such that
\begin{align}\label{eq:lemma_pertubation1}
&\|\tilde{\bV}-\bV\|_F\leq 4\|\bW\|\sqrt{\frac{r}{m}},  \,\,\,\max_{1\leq i\leq m}\|[\bW\tilde{\bV}]_i\|_F\leq \max_{1\leq i\leq m}\|[\bW{\bV}]_i\|_F+4\|\bW\|^2\sqrt{\frac{r}{m}}, \\ &\max_{i=1,\cdots m}\|\tilde{\bV}_i-\bV_i\|_F\leq \frac{2\left(\max_{1\leq i\leq m}\|[\bW{\bV}]_i\|_F+4\|\bW\|^2\sqrt{\frac{r}{m}}+\sqrt{r}\right)}{m-4\|\bW\|\sqrt{r}}.\label{eq:lemma_pertubation2}
\end{align}

\end{lemma}
\begin{proof}[Proof of Theorem~\ref{thm:main}]
Lemma~\ref{lemma:equavalency}  and Lemma~\ref{lemma:noisy2} implies that to prove Theorem~\ref{thm:main}, it is sufficient to prove \eqref{eq:noisy2_0}, which can be verified by applying  Lemma~\ref{lemma:pertubation}.
\end{proof}

\subsection{Proofs of Lemmas}\label{sec:proof_lemma}
\begin{proof}[Proof of Lemma~\ref{lemma:equivalent}]
Since $\bU$ is symmetric and has rank $r$, there exists a decomposition such that $\bU=\bO\bO^T$, where $\bO\in\reals^{D\times r}$. Then the condition $\bU_{ii}\psdle\bI$ implies that all singular values of $\bO_i$ are not greater than $1$, $\tr(\bU_{ii})=r$ implies that $\|\bO_i\|_F=\sqrt{r}$. Recall that $\bO_i\in\reals^{m_i\times r}$ has at most $r$ singular values, we have that all singular values of $\bO_i$ are $1$ and $\bO_i^T\bO_i=\bI$. 

On the other hand, for any $\bO$ in the constraint set of \eqref{eq:tracesum_problem}, $\bU=\bO\bO^T$ lies in the constraint set of \eqref{eq:tracesum_problem1}. 

In addition, with $\bU=\bO\bO^T$ and $\bS_{ii}=\mathbf{0}$, we have $\tr(\bS\bU)=\sum_{i,j=1, i\neq j}^m\tr(\bO_i^T\bS_{ij}\bO_j)$. As a result,  \eqref{eq:tracesum_problem} and  \eqref{eq:tracesum_problem1} are equivalent. 
\end{proof}
	
\begin{proof}[Proof of Lemma~\ref{lemma:stationary}]
Let us consider the tangent space of the constraint set in \eqref{eq:tracesum_problem1} at $\bU=\tilde{\bU}$. In particular, the condition $\bU\psdge \mathbf{0}$ and $\rank(\bU)=r$ give a tangent cone of $\{\bX: P_{\tilde{\bV}^\perp}^T\bX P_{\tilde{\bV}^\perp}=\mathbf{0}\}.$
The condition $\bU_{ii}\psdle \bI$ and $\tr(\bU_{ii})=r$ give the tangent cone of $
\{\bX: P_{\bV_i}^T\bX_{ii}P_{\bV_i}\psdle \mathbf{0}, \tr(\bX_{ii})=0,\,\,\text{for all $1\leq i\leq m$}\}.$

In combination, the tangent cone of \eqref{eq:tracesum_problem1} at $\bU=\tilde{\bU}$ is
\begin{equation}\label{eq:tangent}
\calT_1=\{\bX\in\reals^{D\times D}: P_{\tilde{\bV}^\perp}^T\bX P_{\tilde{\bV}^\perp}=\mathbf{0}, \text{and for all $1\leq i\leq m$,}\,\, P_{\tilde{\bV}_i}^T\bX_{ii}P_{\tilde{\bV}_i}\psdle \mathbf{0}, \tr(\bX_{ii})=0\}.
\end{equation}
Next, we will prove that the set in \eqref{eq:tangent} can simplified to
\begin{equation}\label{eq:tangent1}
\calT_2=\{\bX\in\reals^{D\times D}: P_{\tilde{\bV}^\perp}^T\bX P_{\tilde{\bV}^\perp}=\mathbf{0}, \text{and for all $1\leq i\leq m$,}\,\, P_{\tilde{\bV}_i}^T\bX_{ii}P_{\tilde{\bV}_i}= \mathbf{0}\}.
\end{equation}
Clearly, $\calT_2\subseteq\calT_1$. Next, we will prove $\calT_1\subseteq\calT_2$ by showing that all $\bX\in \calT_1$ satisfies $P_{\tilde{\bV}_i}^T\bX_{ii}P_{\tilde{\bV}_i}= \mathbf{0}$.

For any $\bX\in \calT_1$, since $P_{\tilde{\bV}^\perp}^T\bX P_{\tilde{\bV}^\perp}=\mathbf{0}$, there exists $\bY\in\reals^{D\times r}$ such that $\bX=\bV\bY^T+\bY\bV^T$, and we have $\bX_{ii}=\bV_i\bY_i^T+\bY_i\bV_i^T$. This implies that \[
P_{\tilde{\bV}_i^\perp}^T\bX_{ii}P_{\tilde{\bV}_i^\perp}=\mathbf{0},
\]
and then
\[
0=\tr(\bX_{ii})=\tr(P_{\tilde{\bV}_i^\perp}^T\bX_{ii}P_{\tilde{\bV}_i^\perp})+\tr(P_{\tilde{\bV}_i}^T\bX_{ii}P_{\tilde{\bV}_i})=\tr(P_{\tilde{\bV}_i}^T\bX_{ii}P_{\tilde{\bV}_i}).
\]
Combining it with the assumption that $P_{\tilde{\bV}_i}^T\bX_{ii}P_{\tilde{\bV}_i}\psdge \mathbf{0}$, we have $\tr(P_{\tilde{\bV}_i}^T\bX_{ii}P_{\tilde{\bV}_i})= \mathbf{0}$. This implies that $\bX\in \calT_2$. As a result, $\calT_1=\calT_2$.

Since $\tilde{\bU}$ is the optimal solution for the problem \eqref{eq:tracesum_problem1}, we have $0=\tr(\bX\bS)=\langle\bX, \bS\rangle$ for all $\bX\in\calT_2$. Notice that $\calT_2$ is a subspace in $\reals^{D\times D}$, this implies that $\bS$ lies in the subspace of its orthogonal complement of the subspace, which is 
\begin{equation}\label{eq:tangent3}
\{\bX: \Proj_{\tilde{\bV}^\perp}\bX \Proj_{\tilde{\bV}^\perp}=\bX\}\oplus\{\bX:\bX_{ij}=\mathbf{0} \,\,\text{if $i\neq j$, and $\Proj_{\tilde{\bV}_i}\bX_{ii} \Proj_{\tilde{\bV}_i}=\bX_{ii}$.}\}
\end{equation}
This proves the decomposition of $\bS=\bT^{(1)}+\bT^{(2)}$, such that 
\[
\text{$\Proj_{\tilde{\bV}^\perp}\bT^{(1)} \Proj_{\tilde{\bV}^\perp}=\bT^{(1)}$, $\bT^{(2)}_{ij}=\mathbf{0}$ for $i\neq j$, and $\Proj_{\tilde{\bV}_i}\bT^{(2)}_{ii} \Proj_{\tilde{\bV}_i}=\bT^{(2)}_{ii}$.}
\]
The formula in \eqref{eq:tildeT1} and \eqref{eq:tildeT2} follow from the properties of $\bT^{(1)}, \bT^{(2)}$ and in particular, the last equation in \eqref{eq:tildeT2} follows from the fact that $\bT^{(1)}\tilde{\bV}=\mathbf{0}$.
\end{proof}

\begin{proof}[Proof of Lemma~\ref{lemma:equavalency}]
For any $\bU$ in the constraint set of  \eqref{eq:tracesum_relaxed} such that $\bU\neq \tilde{\bU}$, and $\bX=\bU-\tilde{\bU}$, we have $P_{\tilde{\bV}^\perp}^T\bX P_{\tilde{\bV}^\perp}=P_{\tilde{\bV}^\perp}^T\bU P_{\tilde{\bV}^\perp}-P_{\tilde{\bV}^\perp}^T\tilde{\bU} P_{\tilde{\bV}^\perp}= P_{\tilde{\bV}^\perp}^T\bU P_{\tilde{\bV}^\perp}\psdge \mathbf{0}$, and $P_{\tilde{\bV}_i}^T\bX_{ii}P_{\tilde{\bV}_i}=P_{\tilde{\bV}_i}^T\bU_{ii}P_{\tilde{\bV}_i}-P_{\tilde{\bV}_i}^T\tilde{\bU}_{ii}P_{\tilde{\bV}_i}=P_{\tilde{\bV}_i}^T\bU_{ii}P_{\tilde{\bV}_i}-\bI\psdle \mathbf{0}$. In summary, $\bX$ has the properties of
\begin{equation}\label{eq:tangent4}
P_{\tilde{\bV}^\perp}^T\bX P_{\tilde{\bV}^\perp}\psdge \mathbf{0}, \text{$\tr(\bX_{ii})=0$ and  $P_{\tilde{\bV}_i}^T\bX_{ii}P_{\tilde{\bV}_i}\psdle \mathbf{0}$ for all $1\leq i\leq m$.} 
\end{equation}

In addition, either $P_{\tilde{\bV}^\perp}^T\bX P_{\tilde{\bV}^\perp}$ is nonzero or  $P_{\tilde{\bV}_i}^T\bX_{ii} P_{\tilde{\bV}_i}$ is nonzero. If they are all zero matrices, then we have
\begin{equation}\label{eq:tangent5}
P_{\tilde{\bV}^\perp}^T\bU P_{\tilde{\bV}^\perp}=P_{\tilde{\bV}^\perp}^T\tilde{\bU} P_{\tilde{\bV}^\perp}=\mathbf{0},
\end{equation}
and
\begin{equation}\label{eq:tangent6}
P_{\tilde{\bV}_i}^T\bU_{ii} P_{\tilde{\bV}_i}=P_{\tilde{\bV}_i}^T\tilde{\bU}_{ii} P_{\tilde{\bV}_i}=\bI.
\end{equation}
Since  $\bU_{ii}\psdge 0$, we have $\tilde{\bV}_i\bU_{ii}\tilde{\bV}_i\psdge \mathbf{0}$. Combining it with $\tr(P_{\tilde{\bV}_i}^T\bU_{ii} P_{\tilde{\bV}_i})=r$ and $r=\tr(\bU_{ii})=\tr(P_{\tilde{\bV}_i}^T\bU_{ii} P_{\tilde{\bV}_i})+\tr(\tilde{\bV}_i\bU_{ii}\tilde{\bV}_i)$, we have $\tilde{\bV}_i\bU_{ii}\tilde{\bV}_i= \mathbf{0}$. Combining it with $\bU_{ii}\psdge \mathbf{0}$, we have $\tilde{\bV}_i\bU_{ii}=\mathbf{0}$ and $\bU_{ii}\tilde{\bV}_i=\mathbf{0}$. It implies that 
\begin{equation}\label{eq:tangent7}
\bU_{ii}=\Proj_{\tilde{\bV}_i}^T\bU_{ii} \Proj_{\tilde{\bV}_i}=\tilde{\bV}_i\tilde{\bV}_i^T.
\end{equation}
In addition, \eqref{eq:tangent5} and $\bU\psdge \mathbf{0}$ means that the $\bU=\Proj_{\tilde{\bV}}\bU \Proj_{\tilde{\bV}}$, that is, there exists a matrix $\bZ\in\mathbb{R}^{r\times r}$ such that $\bU=\tilde{\bV}\bZ\tilde{\bV}^T$ and as a result, $\bU_{ii}=\tilde{\bV}_i\bZ\tilde{\bV}_i^T$. Combining it with \eqref{eq:tangent7}, we have $\bZ=\bI$ and $\bU=\tilde{\bV}\tilde{\bV}^T=\tilde{\bU}$, which is a contradiction to $\bU\neq \tilde{\bU}$.

Combining the property of $\bX$ in \eqref{eq:tangent4} with the assumption that $\{P_{\tilde{\bV}_i}^T\hat{\bT}^{(2)}_{ii} P_{\tilde{\bV}_i}\}_{i=1}^m$ and $-P_{\tilde{\bV}^\perp}^T\hat{\bT}^{(1)}P_{\tilde{\bV}^\perp}$ are  positive definite matrices, we have
\begin{align}
&\tr(\bX\bS)=\tr(\bX\hat{\bT}^{(1)})+\tr(\bX\hat{\bT}^{(2)})+c\tr(\bX)= \tr\left((P_{\tilde{\bV}^\perp}^T\bX P_{\tilde{\bV}^\perp})( P_{\tilde{\bV}^\perp}^T\hat{\bT}^{(1)}P_{\tilde{\bV}^\perp})\right)+\sum_{i=1}^m\tr(\bX_{ii} \hat{\bT}^{(2)}_{ii})\\= &\tr\left((P_{\tilde{\bV}^\perp}^T\bX P_{\tilde{\bV}^\perp})( P_{\tilde{\bV}^\perp}^T\hat{\bT}^{(1)}P_{\tilde{\bV}^\perp})\right) + \sum_{i=1}^m\tr\left((P_{\tilde{\bV}_i}^T\bX_{ii} P_{\tilde{\bV}_i})(P_{\tilde{\bV}_i}^T\hat{\bT}^{(2)}_{ii} P_{\tilde{\bV}_i})\right)< 0.\label{eq:tangent8}
\end{align}
The last inequality is strict because either $P_{\tilde{\bV}^\perp}^T\bX P_{\tilde{\bV}^\perp}$ is nonzero or  $P_{\tilde{\bV}_i}^T\bX_{ii} P_{\tilde{\bV}_i}$ is nonzero for some $1\leq i\leq m$. \eqref{eq:tangent} then implies that $\trace(\bS\bU)< \trace(\bS\tilde{\bU})$ for all $\bU\neq\tilde{\bU}$, and as a result, $\tilde{\bU}$ is the unique solution to  \eqref{eq:tracesum_relaxed}.

\end{proof}

\begin{proof}[Proof of Lemma~\ref{lemma:clean}]
Since $L_1\in L_2$, we have $\Proj_{L_2\cap L_1^\perp}=\Proj_{L_2}-\Proj_{L_1}$. Applying the definitions of $L_1$ and $L_2$, we have
\[
[\Proj_{L_1}]_{ij}=\frac{1}{m}\bV_i\bV_j^T
\]
and
\[
[\Proj_{L_2}]_{ii}=\bV_i\bV_i^T, \,\,[\Proj_{L_2}]_{ij}=\mathbf{0}\,\,\text{for $i\neq j$}.
\]
Combining it with the definition of $\bT^{(1)*}$ in \eqref{eq:tildeT1} and \eqref{eq:tildeT2}, we have
\begin{equation}\label{eq:tildeTstar1}
[\bT^{(1)}]_{ij}=\bV_i\bV_i^T\,\,\text{for $i\neq j$}, \,\,P_{\tilde{\bV}_i^\perp}^T[\bT^{(1)}]_{ii}=0,\,\,[\bT^{(1)}]_{ii}=-(m-1){\bV}_i^T{\bV}_i=-[\bT^{(2)}]_{ii},
\end{equation} 
and Lemma~\ref{lemma:clean} is proved.
\end{proof}

\begin{proof}[Proof of Lemma~\ref{lemma:noisy}]
(1)
Applying \eqref{eq:tildeT1} and \eqref{eq:tildeT2}, it is easy to see that \begin{equation}\text{$[\hat{\bT}^{(1)}-\hat{\bT}^{(1)*}]_{ij}=\bW_{ij}$ for $i\neq j$, $[\hat{\bT}^{(1)}]_{ii}=-(\sum_{j=1,j\neq i}^{m}\bS_{ij}\tilde{\bV}_j)\tilde{\bV}_i^T$.}\label{eq:tildeTstar2}\end{equation}Combining it with \eqref{eq:tildeTstar1}, we have
\begin{align}\label{eq:pertubation_1}
\|\hat{\bT}^{(1)}-\hat{\bT}^{(1)*}\|=\|\bT^{(1)}-\bT^{(1)*}\|\leq \|\bW\|+\max_{1\leq i\leq m}\left\|\left(\sum_{j=1,j\neq i}^{m}\bS_{ij}\tilde{\bV}_j\right)\tilde{\bV}_i^T-(m-1)\bV_i\bV_i^T\right\|.
\end{align}

Applying 
\begin{equation}\label{eq:j_exclude_i}
\|(\sum_{j=1,j\neq i}^{m}\bV_j^T\tilde{\bV}_j)-(m-1)\bI\|\leq \|(\sum_{j=1}^{m}\bV_j^T\tilde{\bV}_j)-m\bI\|+\|\bV_i^T\tilde{\bV}_i-\bI\|\leq \|\bV^T\tilde{\bV}-m\bI\|+\|\bV_i-\tilde{\bV}_i\|,
\end{equation}
we have 
\begin{align}\label{eq:j_exclude_i2}
&\|(\sum_{j=1,j\neq i}^{m}\bV_i\bV_j^T\tilde{\bV}_j)\tilde{\bV}_i^T-(m-1)\bV_i\bV_i^T\|
= \|(\sum_{j=1,j\neq i}^{m}\bV_j^T\tilde{\bV}_j)\tilde{\bV}_i^T-(m-1)\bV_i^T\|
\\\leq &\|(\sum_{j=1,j\neq i}^{m}\bV_j^T\tilde{\bV}_j)-(m-1)\bI\|+(m-1)\|\tilde{\bV}_i-\bV_i\|\leq \|\bV^T\tilde{\bV}-m\bI\|+m\|\tilde{\bV}_i-\bV_i\|.\nonumber
\end{align}
Applying \eqref{eq:j_exclude_i2} and $\bS_{ij}=\bW_{ij}+\bV_i\bV_j^T$ when $i\neq j$, the RHS of \eqref{eq:pertubation_1} is bounded by $m\max_{1\leq i\leq m}\|\tilde{\bV}_i-\bV_i\|+\max_{1\leq i\leq m}\|\sum_{j=1}^m\bW_{ij}\tilde{\bV}_j\|+\|\bV^T\tilde{\bV}-m\bI\|.$

(2)  Applying \eqref{eq:tildeTstar2} and \eqref{eq:tildeT2}, we have \[
[\hat{\bT}^{(2)}]_{ii}=(\sum_{j=1,j\neq i}^{m}\bS_{ij}\tilde{\bV}_j)\tilde{\bV}_i^T.
\]
Combining it with \eqref{eq:j_exclude_i}, 
\begin{align*}
&\|\tilde{\bV}_i^T[\bT^{(2)}]_{ii}\tilde{\bV}_i-(m-1)\bI\|=\left\|\tilde{\bV}_i^T\Big(\sum_{j=1,j\neq i}^m\bS_{ij}\tilde{\bV}_j\Big)-(m-1)\bI\right\|\\\leq& \|\sum_{j=1}^m\bW_{ij}\tilde{\bV}_j\|+\left\|\tilde{\bV}_i^T\bV_i \Big(\sum_{j=1,j\neq i}^m{\bV}_j^T\tilde{\bV}_j\Big)-(m-1)\bI\right\|\\\leq &\Big\|\sum_{j=1}^m\bW_{ij}\tilde{\bV}_i\Big\|+\|\tilde{\bV}_i^T\bV_i-\bI\|\Big\|\sum_{j=1,j\neq i}^m{\bV}_j^T\tilde{\bV}_j\Big\|+\left\|\sum_{j=1,j\neq i}^m{\bV}_j^T\tilde{\bV}_j-(m-1)\bI\right\|
\\\leq &
\Big\|\sum_{j=1}^m\bW_{ij}\tilde{\bV}_i\Big\|+(m-1)\|\tilde{\bV}_i^T\bV_i-\bI\|+\|{\bV}^T\tilde{\bV}-m\bI\|+\|\bV_i-\tilde{\bV}_i\|\\
=&\Big\|\sum_{j=1}^m\bW_{ij}\tilde{\bV}_i\Big\|+m\|\tilde{\bV}_i^T\bV_i-\bI\|+\|{\bV}^T\tilde{\bV}-m\bI\|.
\end{align*}

\end{proof}

\begin{proof}[Proof of Lemma~\ref{lemma:noisy2}]
Define $\hat{\bT}^{(1)}$ and $\hat{\bT}^{(2)}$ as follows:
\[
\hat{\bT}^{(1)}_{ij}=\bT^{(1)}_{ij}\,\,\text{and $\hat{\bT}^{(2)}_{ij}=\bT^{2}_{ij}$ for $i\neq j$, $\hat{\bT}^{(1)}_{ii}=\bT^{(1)}_{ii}-c\Proj_{\tilde{\bV}_i^\perp}$, and $\hat{\bT}^{(2)}_{ii}=\bT^{(2)}_{ii}-c\Proj_{\tilde{\bV}_i}.$}
\]
Clearly, $\bS=\hat{\bT}^{(1)}+\hat{\bT}^{(2)}+c\bI$, $\Proj_{\tilde{\bV}^\perp}\hat{\bT}^{(1)}\Proj_{\tilde{\bV}^\perp}=\hat{\bT}^{(1)}$, and $\hat{\bT}^{(2)}_{ii}=\Proj_{\tilde{\bV}_i}\hat{\bT}^{(2)}_{ii} \Proj_{\tilde{\bV}_i}$. It remains to show that $P_{\tilde{\bV}_i}^T\hat{\bT}^{(2)}_{ii} P_{\tilde{\bV}_i}$ and $-P_{\tilde{\bV}^\perp}^T\hat{\bT}^{(1)}P_{\tilde{\bV}^\perp}$ are  positive definite matrices. 

Applying \eqref{eq:noisy1}, we have that $P_{{\tilde{\bV}_i}}^T\hat{\bT}^{(2)}_{ii}P_{{\tilde{\bV}_i}}$ is positive definite if
\begin{equation}\label{eq:noisy2_1}
m-1>c+\Big\|\sum_{j=1}^m\bW_{ij}\tilde{\bV}_i\Big\|+m\|\tilde{\bV}_i^T\bV_i-\bI\|+\|{\bV}^T\tilde{\bV}-m\bI\|.
\end{equation}
If we define the subspace $L_3=L_2^\perp=\{\bx\in\reals^D: \bx_i\in\Sp(\bV_i^\perp)\}$, then we have  $\bT^{(1)*}=-m\Proj_{L_2\cap L_1^\perp}-c\Proj_{L_3}$. Considering that $\dim(L_2\cap L_1^\perp)=\dim(L_2)-\dim(L_1)=rm-r$ and $\dim(L_3)=D-\dim(L_2)=D-rm$, we have $\lambda_{r+1}(\hat{\bT}^{(1)*})=-m/2$. Applying the result on the perturbation of eigenvalue that
\[
|\lambda_{r+1}(\hat{\bT}^{(1)*})-\lambda_{r+1}(\hat{\bT}^{(1)})|\leq \|\hat{\bT}^{(1)*}-\hat{\bT}^{(1)}\|
\]
and \eqref{eq:noisy2}, when
\begin{equation}\label{eq:noisy2_2}c>m\max_{1\leq i\leq m}\|\tilde{\bV}_i-\bV_i\|+\max_{1\leq i\leq m}\|\sum_{j=1}^m\bW_{ij}\tilde{\bV}_j\|+\|\bV^T\tilde{\bV}-m\bI\|,
\end{equation}
$\lambda_{r+1}(\hat{\bT}^{(1)})$ is negative, which means that $\hat{\bT}^{(1)}$ has at least $D-r$ negative eigenvalues. By definition, $\hat{\bT}^{(1)}$ has $r$ zeros eigenvalues with eigenvectors spanning the column space of $\tilde{\bV}$, so the $P_{\tilde{\bV}^\perp}^T\hat{\bT}^{(1)}P_{\tilde{\bV}^\perp}$ is negative definite. 

When \eqref{eq:noisy2_0} holds, then applying \eqref{eq:noisy2}, we can find $c$ such that both \eqref{eq:noisy2_1} and \eqref{eq:noisy2_2} are true, and Lemma~\ref{lemma:noisy2} is then proved.  
\end{proof}

\begin{proof}[Proof of Lemma~\ref{lemma:pertubation}]
First, we remark that the choice of $\tilde{\bV}$ is only unique up to an $r\times r$ orthogonal matrix. In this proof, we choose $\tilde{\bV}$ such that $\tilde{\bV}^T$ is a symmetric, positive semidefinite matrix.

Then we have that
\begin{align}\label{eq:pertubation1}
&\|\tilde{\bV}-\bV\|_F^2=\sum_{i=1}^m\|\tilde{\bV}_i-\bV_i\|_F^2=\sum_{i=1}^m\|\tilde{\bV}_i\|_F^2+\|\bV_i\|_F^2-2\tr(\tilde{\bV}_i\bV_i^T)=\sum_{i=1}^m\|\tilde{\bV}_i\|_F^2+\|\bV_i\|_F^2-2\tr(\bV_i^T\tilde{\bV}_i)\\=&2rm-2\tr(\sum_{i=1}^m\bV_i^T\tilde{\bV}_i)=2rm-2\|\bV^T\tilde{\bV}\|_*,
\end{align}
where $\|\cdot\|_*$ represents the nuclear norm that is the summation of all singular values (and since $\bV^T\tilde{\bV}$ is positive semidefinite, it is also the summation of its eigenvalues).

Using the definition in \eqref{eq:tracesum_problem1}, we have
\begin{align}\label{eq:pertubation2}
\tr(\tilde{\bV}^T\bS\tilde{\bV})\geq 
\tr({\bV}^T\bS{\bV}).
\end{align}
With the definition of $\bS$ (note that if $\bW=0$, then $\tr(\tilde{\bV}^T\bS\tilde{\bV})=\|\tilde{\bV}^T\bV\|_F^2-rm$), it implies that
\begin{align}\label{eq:pertubation4}
\tr(\tilde{\bV}^T\bW\tilde{\bV})-\tr({\bV}^T\bW{\bV})\geq \|\bV^T\bV\|_F^2-\|\tilde{\bV}^T\bV\|_F^2=rm^2-\|\tilde{\bV}^T\bV\|_F^2.
\end{align}
Since $\|\bX\|_F^2=\sum_{i}\lambda_i(\bX)^2$, we have
\begin{align}\label{eq:pertubation5}
rm^2-\|\tilde{\bV}^T\bV\|_F^2=\sum_{i=1}^r(m^2-\lambda_i(\tilde{\bV}^T\bV)^2)\geq m \sum_{i=1}^r(m-\lambda_i(\tilde{\bV}^T\bV))=m(rm-\|\tilde{\bV}^T\bV\|_*).
\end{align}
The combination \eqref{eq:pertubation4}, \eqref{eq:pertubation5}, $\|\tilde{\bV}\|_F=\|\bV\|_F=\sqrt{rm}$,  $\tr(\bA\bB)\leq \|\bA\|_F\|\bB\|_F$, and $\|\bA\bB\|_F\leq \|\bA\|\|\bB\|_F$ implies that
\begin{align*}
&m(rm-\|\tilde{\bV}^T\bV\|_*)\leq \tr(\tilde{\bV}^T\bW\tilde{\bV})-\tr({\bV}^T\bW{\bV})
=\tr((\tilde{\bV}-\bV)^T\bW\tilde{\bV})+\tr(\bV^T\bW(\tilde{\bV}-\bV))\\
\leq& \|\bW\|\|\tilde{\bV}-\bV\|_F\|\tilde{\bV}\|_F+\|\bW\|\|\tilde{\bV}-\bV\|_F\|{\bV}\|_F=2\|\bW\|\|\tilde{\bV}-\bV\|_F\sqrt{rm}.
\end{align*}
Combining it with \eqref{eq:pertubation1}, we have
\begin{align}\label{eq:pertubation6}
\frac{m}{2}\|\tilde{\bV}-\bV\|_F^2\leq 2\|\bW\|\|\tilde{\bV}-\bV\|_F\sqrt{rm},
\end{align}
which implies that
\begin{align}\label{eq:pertubation7}
\|\tilde{\bV}-\bV\|_F\leq 4\|\bW\|\sqrt{\frac{r}{m}},
\end{align}
which proves the first inequality in \eqref{eq:lemma_pertubation1}. It implies that
\begin{align}\label{eq:pertubation7.5}
\|\tilde{\bV}\bV^T-m\bI\|_F=\|(\tilde{\bV}-\bV)\bV^T\|_F\leq \|\tilde{\bV}-\bV\|_F\|\bV\|= \|\tilde{\bV}-\bV\|_F\sqrt{m}\leq 4\|\bW\|\sqrt{r}.
\end{align}

Now let us consider $\bar{\bV}\in\reals^{D\times r}$ defined by $\bar{\bV}_i=\bV_i$ and $\bar{\bV}_j=\tilde{\bV}_j$ for all $1\leq j\leq m, j\neq i$. By definition we have
\[
\tr(\tilde{\bV}^T\bS\tilde{\bV})\geq 
\tr(\bar{\bV}^T\bS\bar{\bV}),
\]
and it is equivalent to
\[
\tr((\tilde{\bV}-\bar{\bV})^T\bS\tilde{\bV})+
\tr(\tilde{\bV}\bS(\tilde{\bV}-\bar{\bV}))-\tr((\tilde{\bV}-\bar{\bV})^T\bS(\tilde{\bV}-\bar{\bV}))\geq 0.
\]
By the definition of $\bar{\bV}$, $\tilde{\bV}$, and notice that $\bS_{ii}=\mathbf{0}$, we have
\[
2\tr\left((\bV_i-\tilde{\bV}_i)^T[\bS\tilde{\bV}]_i\right)\leq 0.
\]
By the definition of $\bS$, it implies that
\begin{align}\label{eq:pertubation10}
\tr\left((\bV_i-\tilde{\bV}_i)^T\bV_i\bV^T\tilde{\bV}- (\bV_i-\tilde{\bV}_i)^T\bV_i\bV_i^T\tilde{\bV}_i+(\bV_i-\tilde{\bV}_i)^T[\bW\tilde{\bV}]_i\right)\leq 0
\end{align}
Recall that $\bV^T\tilde{\bV}$ is symmetric, positive semidefinite, and apply the fact that when $\bA$ is positive semidefinite, then $\tr(\bB\bA)=\tr(\bB^T\bA)$ and when both $\bA, \bB$ are p.s.d., $\tr(\bA\bB)\geq \tr(\bA\lambda_{\min}(\bB)\bI)\geq \lambda_{\min}(\bB)\tr(\bA)$ ($\lambda_{\min}$ represents the smallest eigenvalue), we have
\begin{align}\label{eq:pertubation11}
&\tr\left((\bV_i-\tilde{\bV}_i)^T\bV_i\bV^T\tilde{\bV}\right)=
\tr\left((\bI-\tilde{\bV}_i^T\bV_i)(\bV^T\tilde{\bV})\right)=
\frac{1}{2}\tr\left((2\bI-\tilde{\bV}_i^T\bV_i-\bV_i^T\tilde{\bV}_i)(\bV^T\tilde{\bV})\right)\\=&\frac{1}{2}\tr\left((\tilde{\bV}_i-\bV_i)^T(\tilde{\bV}_i-\bV_i)(\bV^T\tilde{\bV})\right)\geq \frac{1}{2}\tr\left((\tilde{\bV}_i-\bV_i)^T(\tilde{\bV}_i-\bV_i)\right)\lambda_r(\bV^T\tilde{\bV})=\frac{1}{2}\|\tilde{\bV}_i-\bV_i\|_F^2\lambda_r(\bV^T\tilde{\bV}).\nonumber
\end{align}
Combining \eqref{eq:pertubation10}, \eqref{eq:pertubation11}, and $\tr(\bA\bB)\leq \|\bA\|_F\|\bB\|_F,$ we have
\[\frac{1}{2}\lambda_r(\bV^T\tilde{\bV})\|\bV_i-\tilde{\bV}_i\|_F^2\leq \|\bV_i-\tilde{\bV}_i\|_F(\|[\bW\tilde{\bV}]_i\|_F+\sqrt{r})\]
Since it holds for all $1\leq i\leq m$, 
\begin{align}\label{eq:pertubation12}
\frac{1}{2}\lambda_r(\bV^T\tilde{\bV})\max_{1\leq i\leq m}\|\bV_i-\tilde{\bV}_i\|_F\leq \max_{1\leq i\leq m}\|[\bW\tilde{\bV}]_i\|_F+\sqrt{r}.\end{align}
Applying \eqref{eq:pertubation7}, the second inequality in \eqref{eq:lemma_pertubation1} is proved:
\begin{align}\nonumber&\max_{1\leq i\leq m}\|[\bW\tilde{\bV}]_i\|_F\leq \max_{1\leq i\leq m}\|[\bW{\bV}]_i\|_F+\max_{1\leq i\leq m}\|[\bW(\tilde{\bV}-\bV)]_i\|_F\leq \max_{1\leq i\leq m}\|[\bW{\bV}]_i\|_F\\&+\|\bW\|\|\tilde{\bV}-\bV\|_F
\leq \max_{1\leq i\leq m}\|[\bW{\bV}]_i\|_F+4\|\bW\|^2\sqrt{\frac{r}{m}}.\label{eq:pertubation13}
\end{align}
Combining \eqref{eq:pertubation13} with \eqref{eq:pertubation7.5} (which implies that $\lambda_r(\bV^T\tilde{\bV})\geq m-4\|\bW\|\sqrt{r}$), and \eqref{eq:pertubation12}, \eqref{eq:lemma_pertubation2} is proved.
\end{proof}
\section{Conclusion}
This paper studies the orthogonal trace-sum maximization \cite{Won2018}. It shows that while the problem is nonconvex, its solution can be achieved by solving its convex relaxation when the noise is small.

A future direction is to improve the estimation on maximum noise that this method can handle. While this paper showed that the method succeeds when $\sigma=O(m^{1/4})$, we expect that it would also hold for noise as large as $\sigma=O(m^{1/2})$, which has been proved in \cite{doi:10.1137/17M1122025} for the special case of phase synchronization. Another future direction is to use a more general model than \eqref{eq:model}, which would have a larger range of real-life applications.

\section{Acknowledgement}
The author would like to thank Hua Zhou and Joong-Ho Won for introducing this interesting problem and related literatures.

%
%
%
\bibliographystyle{abbrv}
\bibliography{bib-online}

\begin{thebibliography}{10}

\bibitem{abbe2017group}
E.~Abbe, L.~Massoulie, A.~Montanari, A.~Sly, and N.~Srivastava.
\newblock Group synchronization on grids, 2017.

\bibitem{6374980}
M.~{Arie-Nachimson}, S.~Z. {Kovalsky}, I.~{Kemelmacher-Shlizerman},
  A.~{Singer}, and R.~{Basri}.
\newblock Global motion estimation from point matches.
\newblock In {\em 2012 Second International Conference on 3D Imaging, Modeling,
  Processing, Visualization Transmission}, pages 81--88, Oct 2012.

\bibitem{7785130}
F.~{Arrigoni}, A.~{Fusiello}, and B.~{Rossi}.
\newblock Camera motion from group synchronization.
\newblock In {\em 2016 Fourth International Conference on 3D Vision (3DV)},
  pages 546--555, Oct 2016.

\bibitem{Bandeira2017}
A.~S. Bandeira, N.~Boumal, and A.~Singer.
\newblock Tightness of the maximum likelihood semidefinite relaxation for
  angular synchronization.
\newblock {\em Mathematical Programming}, 163(1):145--167, May 2017.

\bibitem{Bandeira20164}
A.~S. Bandeira, C.~Kennedy, and A.~Singer.
\newblock Approximating the little grothendieck problem over the orthogonal and
  unitary groups.
\newblock {\em Mathematical Programming}, 160(1):433--475, Nov 2016.

\bibitem{measure2005}
A.~Barvinok.
\newblock {Math 710: Measure Concentration}, 2005.

\bibitem{doi:10.1137/16M105808X}
N.~Boumal.
\newblock Nonconvex phase synchronization.
\newblock {\em SIAM Journal on Optimization}, 26(4):2355--2377, 2016.

\bibitem{DBLP:journals/corr/BoumalC13}
N.~Boumal and X.~Cheng.
\newblock Expected performance bounds for estimation on graphs from random
  relative measurements.
\newblock {\em CoRR}, abs/1307.6398, 2013.

\bibitem{6760038}
N.~{Boumal}, A.~{Singer}, and P.~. {Absil}.
\newblock Robust estimation of rotations from relative measurements by maximum
  likelihood.
\newblock In {\em 52nd IEEE Conference on Decision and Control}, pages
  1156--1161, Dec 2013.

\bibitem{10.1093/imaiai/iat006}
N.~Boumal, A.~Singer, P.-A. Absil, and V.~D. Blondel.
\newblock {Cram\'{e}r?Rao bounds for synchronization of rotations}.
\newblock {\em Information and Inference: A Journal of the IMA}, 3(1):1--39, 09
  2013.

\bibitem{doi:10.1137/130935458}
K.~N. Chaudhury, Y.~Khoo, and A.~Singer.
\newblock Global registration of multiple point clouds using semidefinite
  programming.
\newblock {\em SIAM Journal on Optimization}, 25(1):468--501, 2015.

\bibitem{doi:10.1002/cpa.21760}
Y.~Chen and E.~J. Cand{\`e}s.
\newblock The projected power method: An efficient algorithm for joint
  alignment from pairwise differences.
\newblock {\em Communications on Pure and Applied Mathematics},
  71(8):1648--1714, 2018.

\bibitem{6875186}
Y.~{Chen} and A.~J. {Goldsmith}.
\newblock Information recovery from pairwise measurements.
\newblock In {\em 2014 IEEE International Symposium on Information Theory},
  pages 2012--2016, June 2014.

\bibitem{10.1093/comnet/cnu050}
M.~Cucuringu.
\newblock {Synchronization over Z2 and community detection in signed multiplex
  networks with constraints}.
\newblock {\em Journal of Complex Networks}, 3(3):469--506, 01 2015.

\bibitem{Cucuringu2012}
M.~Cucuringu, A.~Singer, and D.~Cowburn.
\newblock {Eigenvector synchronization, graph rigidity and the molecule
  problem}.
\newblock {\em Information and inference : a journal of the IMA}, 1(1):21, dec
  2012.

\bibitem{Szarek:survey}
K.~Davidson and S.~Szarek.
\newblock Local operator theory, random matrices and {B}anach spaces.
\newblock In Lindenstrauss, editor, {\em Handbook on the Geometry of Banach
  spaces}, volume~1, pages 317--366. Elsevier Science, 2001.

\bibitem{Fei2019AchievingTB}
Y.~Fei and Y.~Chen.
\newblock Achieving the bayes error rate in synchronization and block models by
  sdp, robustly.
\newblock {\em ArXiv}, abs/1904.09635, 2019.

\bibitem{Gao2019}
T.~Gao, J.~Brodzki, and S.~Mukherjee.
\newblock The geometry of synchronization problems and learning group actions.
\newblock {\em Discrete {\&} Computational Geometry}, May 2019.

\bibitem{pmlr-v97-gao19f}
T.~Gao and Z.~Zhao.
\newblock Multi-frequency phase synchronization.
\newblock In K.~Chaudhuri and R.~Salakhutdinov, editors, {\em Proceedings of
  the 36th International Conference on Machine Learning}, volume~97 of {\em
  Proceedings of Machine Learning Research}, pages 2132--2141, Long Beach,
  California, USA, 09--15 Jun 2019. PMLR.

\bibitem{7430328}
Y.~{Khoo} and A.~{Kapoor}.
\newblock Non-iterative rigid 2d/3d point-set registration using semidefinite
  programming.
\newblock {\em IEEE Transactions on Image Processing}, 25(7):2956--2970, July
  2016.

\bibitem{doi:10.1137/18M1217644}
S.~Ling, R.~Xu, and A.~S. Bandeira.
\newblock On the landscape of synchronization networks: A perspective from
  nonconvex optimization.
\newblock {\em SIAM Journal on Optimization}, 29(3):1879--1907, 2019.

\bibitem{doi:10.1137/16M110109X}
H.~Liu, M.-C. Yue, and A.~Man-Cho~So.
\newblock On the estimation performance and convergence rate of the generalized
  power method for phase synchronization.
\newblock {\em SIAM Journal on Optimization}, 27(4):2426--2446, 2017.

\bibitem{doi:10.1137/16M1106055}
O.~\"{O}zyesil, N.~Sharon, and A.~Singer.
\newblock Synchronization over cartan motion groups via contraction.
\newblock {\em SIAM Journal on Applied Algebra and Geometry}, 2(2):207--241,
  2018.

\bibitem{doi:10.1002/cpa.21750}
A.~Perry, A.~S. Wein, A.~S. Bandeira, and A.~Moitra.
\newblock Message-passing algorithms for synchronization problems over compact
  groups.
\newblock {\em Communications on Pure and Applied Mathematics},
  71(11):2275--2322, 2018.

\bibitem{Singer2009AngularSB}
A.~Singer.
\newblock Angular synchronization by eigenvectors and semidefinite programming.
\newblock {\em Applied and computational harmonic analysis}, 30:20--36, 01
  2011.

\bibitem{doi:10.1137/090767777}
A.~Singer and Y.~Shkolnisky.
\newblock Three-dimensional structure determination from common lines in
  cryo-em by eigenvectors and semidefinite programming.
\newblock {\em SIAM Journal on Imaging Sciences}, 4(2):543--572, 2011.

\bibitem{THUNBERG2017243}
J.~Thunberg, F.~Bernard, and J.~Goncalves.
\newblock Distributed methods for synchronization of orthogonal matrices over
  graphs.
\newblock {\em Automatica}, 80:243 -- 252, 2017.

\bibitem{7789623}
R.~{Tron}, X.~{Zhou}, and K.~{Daniilidis}.
\newblock A survey on rotation optimization in structure from motion.
\newblock In {\em 2016 IEEE Conference on Computer Vision and Pattern
  Recognition Workshops (CVPRW)}, pages 1032--1040, June 2016.

\bibitem{10.1093/imaiai/iat005}
L.~Wang and A.~Singer.
\newblock {Exact and stable recovery of rotations for robust synchronization}.
\newblock {\em Information and Inference: A Journal of the IMA}, 2(2):145--193,
  12 2013.

\bibitem{Won2018}
J.-H. Won, H.~Zhou, and K.~Lange.
\newblock {Orthogonal Trace-Sum Maximization: Applications, Local Algorithms,
  and Global Optimality}.
\newblock nov 2018.

\bibitem{ZHANG2017159}
T.~Zhang and A.~Singer.
\newblock Disentangling orthogonal matrices.
\newblock {\em Linear Algebra and its Applications}, 524:159 -- 181, 2017.

\bibitem{doi:10.1137/17M1122025}
Y.~Zhong and N.~Boumal.
\newblock Near-optimal bounds for phase synchronization.
\newblock {\em SIAM Journal on Optimization}, 28(2):989--1016, 2018.

\end{thebibliography}
\end{document}